\documentclass[11pt]{article}

\usepackage[margin=1.27in]{geometry}

\usepackage[misc]{ifsym} % for letter symbol
\usepackage{lipsum}
\usepackage{amssymb}
\usepackage{amsmath}
\usepackage{latexsym}
\usepackage{bm}
\usepackage{enumitem}
\usepackage{graphicx}
\usepackage{amscd}
\usepackage{extarrows}
\usepackage{amsfonts}
\usepackage{amssymb}
\usepackage{indentfirst}
\usepackage{bbm}
\usepackage{mathdots}
\usepackage{mathtools}
\usepackage[linktocpage]{hyperref}
\usepackage{amsthm}
\usepackage{ upgreek }
\hypersetup{
    colorlinks,
    citecolor=black,
    filecolor=black,
    linkcolor=blue,
    urlcolor=black
}

\usepackage{quiver}
\usepackage[multiple]{footmisc}
\usepackage[title]{appendix}

\newcommand\blfootnote[1]{%
  \begingroup
  \renewcommand\thefootnote{}\footnote{#1}%
  \addtocounter{footnote}{-1}%
  \endgroup
}

\usepackage{float}

\theoremstyle{plain}
\newtheorem{theorem}{Theorem}[section]
\newtheorem{proposition}[theorem]{Proposition}
\newtheorem{lemma}[theorem]{Lemma}

\newtheorem*{theorem*}{Theorem}
\newtheorem*{lemma*}{Lemma}

\theoremstyle{remark}
\newtheorem{remark}[theorem]{Remark}

\theoremstyle{definition}
\newtheorem{definition}{Definition}[section]

\newcommand*\dd{\mathop{}\!\mathrm{d}}

\DeclareMathOperator{\tr}{tr}

\DeclareMathOperator{\Tr}{Tr}

\DeclareMathOperator{\id}{id}
\DeclareMathOperator{\Ind}{Ind}

\DeclareMathOperator{\HS}{HS}

\DeclareMathOperator{\GL}{GL}

\DeclareMathOperator{\PSL}{PSL}

\DeclarePairedDelimiter\abs{\lvert}{\rvert}

\title{The Atiyah-Schmid formula for reductive groups}
\author{Jun Yang}
\date{}

\begin{document}
\maketitle
\begin{abstract} 
We give the generalized Atiyah-Schmid formula for projective tempered representations. 
Then we prove the Atiyah-Schmid formula for arithmetic subgroups of real reductive groups. 
\end{abstract}
%\tableofcontents

%\blfootnote{AMS 2010 Mathematics Subject Classification: 46L10,  20G05, 20G35. }
\blfootnote{\Letter~ \href{mailto:junyang@fas.harvard.edu}{junyang@fas.harvard.edu}~~~~~~~Harvard University, Cambridge, MA 02138, USA}
%\blfootnote{\href{mailto:junyang@fas.harvard.edu}{junyang@fas.harvard.edu}}
\blfootnote{This work was supported in part by the ARO Grant W911NF-19-1-0302 and the ARO MURI Grant W911NF-20-1-0082.}

\tableofcontents

\section{Introduction}

In the study of discrete series representations of semisimple Lie groups, Atiyah and Schmid proposed a formula connecting the formal degrees of discrete series representations and the dimensions over a discrete subgroup (see \cite[formula 3.3]{AS77} and \cite[Theorem 3.3.2]{GHJ}). 
This article is devoted to large generalizations of such a formula, which work for projective tempered representations and real reductive groups. 

Let $G$ be a semisimple Lie group with a Haar measure $\mu$. 
Let $(\pi,H)$ be a discrete series representation of $G$, which is an irreducible representation whose matrix coefficients belong to $L^2(G,\mu)$. 
Let $d(\pi)$ be its formal degree.  
For a lattice $\Gamma$ of $G$, i.e., a discrete subgroup $\Gamma$ of $G$ such that $\mu(\Gamma\backslash G)$ is finite,  
we let $\mathcal{L}(\Gamma)$ be the left group von Neumann algebra of $\Gamma$ and $\dim_{\mathcal{L}(\Gamma)}H$ be the dimension of $H$ over this algebra. 
Then the Atiyah-Schmid formula is given as
\begin{equation}\label{eAS}
    \dim_{\mathcal{L}(\Gamma)}H=\mu(\Gamma\backslash G)\cdot d(\pi).
\end{equation}
For example, if $G=\PSL(2,\mathbb{R})$, $\Gamma=\PSL(2,\mathbb{Z})$ and $(\pi,H)$ is the holomorphic discrete series representation of the lowest weight, we have $\dim_{\mathcal{L}(\Gamma)}H=\frac{1}{12}$.

We expect an analog of \ref{eAS} for reductive groups such as $\GL(n,\mathbb{R})$. 
But there are some obstacles for such generalization from semisimple groups:
\begin{enumerate}
    \item Reductive groups have no square-integrable irreducible representations but only such representations modulo the center. 
    For instance, if $(\pi,H_{\pi})$ is a subrepresentation of the left regular representation of $G$ on $L^2(G)$, we have
    \begin{center}
        $\langle u,v\rangle_{H_{\pi}}=\int_{G/Z(G)}u(\dot{g})\overline{v(\dot{g})}\left(\int_{Z(G)}1dz\right)dg$ for $u,v\in H_{\pi}$. 
    \end{center}
    This integral diverges since the center $Z(G)$ usually contains a torus and the inner integral is infinite. 
    While $\pi$ can be regarded as a representation $G/Z(G)$, it is actually projective, and formal degrees or Plancherel measures are only related to $G/Z(G)$ instead of to $G$. 

    \item For a real reductive group $G(\mathbb{R})$, the most interesting discrete subgroups are arithmetic groups such as $G(\mathbb{Z})$, all of which fail to be lattices in general. 
    For example, for the group $G=\GL(n)$, the volume of the quotient space $\mu(\GL(\mathbb{Z})\backslash\GL(\mathbb{R}))$ is infinite. 
\end{enumerate}
Both these two problems are solved in this article  by introducing the notion of the {\it $\Gamma$-density} $d_{\Gamma}(H)$ over a discrete group $\Gamma$ for a Hilbert space $H$, which is the analogue of the $\Gamma$-dimension in \cite[\S 1]{Aty76}. 

For a 2-cocycle $\omega$ of $\Gamma$, we  consider the $\omega$-projective representations, which are continuous maps $\pi\colon G\to U(H)$ such that $\pi(g)\pi(h)=\omega(g,h)\pi(gh)$ for all $g,h\in G$.  
We can further define the $(\Gamma,\omega)$-density $d_{\Gamma,\omega}(H)$ for these representations. 
A formula for such representations is first obtained. 
\begin{lemma*}
Let $\Gamma$ be a lattice in a unimodular type I locally compact group $G$. 
Let $\nu_{G,\omega}$ be the Plancherel measure on the $\omega$-projective dual $\Pi(G,\omega)$ of $G$ for a 2-cocycle $\omega$. 
We have
\begin{center}
$d_{\Gamma,\omega}(H_{\pi})=\mu(\Gamma/G)\cdot d\nu_{G,\omega}(\pi)$. 
    %$\dim_{\mathcal{L}(\Gamma,\omega)}(H_{\pi})=\mu(\Gamma/G)\cdot d\nu_{G,\omega}(\pi)$. 
\end{center}
\end{lemma*}

Now we let $G$ be a real reductive group. 
Note that $\overline{G}=G/Z(G)$ is semisimple and its integral points $\overline{G}(\mathbb{Z})$ is a lattice. 
\begin{theorem*}
Let $G=G(\mathbb{R})$ be a real reductive group and $\Gamma=G(\mathbb{Z})$. 
Let $\overline{\Gamma}$ be the image of $\Gamma$ under the quotient map $G\to \overline{G}$. 
We have
\begin{equation*}
d_{\Gamma}(H_{\pi})=\frac{\mu_{\overline{G}}(\overline{\Gamma}/\overline{G})}{|Z\cap \Gamma|}\cdot d\nu_{G}(\pi).
    %\dim_{\mathcal{L}(\Gamma)}H_{\pi}=\frac{\mu_{\overline{G}}(\overline{\Gamma}/\overline{G})}{|Z\cap \Gamma|}\cdot d\nu_{G}(\pi). 
\end{equation*}
\end{theorem*}

In Section \ref{sdenvna}, we quickly review von Neumann dimensions and give the definition of von Neumann densities, which are shown to be a well-defined notion as the analogue of $\Gamma$-dimensions. 
Section \ref{sASproj} is devoted to Theorem \ref{tdimmeas}, which is the Atiyah-Schmid formula extended for projective tempered representations. 
In Section \ref{sASred}, we state and prove the result above (see Theorem \ref{tASred}), which is the Atiyah-Schmid formula that works for reductive groups.

\section{The density over a von Neumann algebra}\label{sdenvna}

Let $\Gamma$ be a countable discrete group and $\{\delta_{\gamma}\}_{\gamma\in \Gamma}$ be the canonical orthonormal basis of $l^2(\Gamma)$. 
We let $\lambda$ and $\rho$ be the left and right regular representations of $\Gamma$ on $l^2(\Gamma)$ respectively.
For all $\gamma,\gamma'\in \Gamma$, we have
%\begin{center}
$\lambda(\gamma')\delta_{\gamma}=\delta_{\gamma'\gamma}$ and $\rho(\gamma')\delta_{\gamma}=\delta_{\gamma\gamma'^{-1}}$. 
%\end{center}
Let $\mathcal{L}(\Gamma)$ be the strong operator closure of the complex linear span of $\lambda(\gamma)$'s.  
This is the {\it left group von Neumann algebra of $\Gamma$}. 

Let $\omega$ be a normalized 2-cocycle of $\Gamma$. 
Let $\lambda_{\omega},\rho_{\omega}$ be the $\omega$-projective left and right regular representation of $\Gamma$ on $l^2(\Gamma)$, which are defined as $\lambda_{\omega}(\gamma)f(x)=\omega(x^{-1},\gamma)f(\gamma^{-1}x)$ and $\rho_{\omega}(\gamma)f(x)=\omega(\gamma^{-1},x^{-1})f(x\gamma)$ for $f\in l^2(\Gamma)$. 
Following \cite{VaDo2024,Enstad22}, we  define
\begin{enumerate}
\item the {\it $\omega$-twisted left group von Neumann algebra} $\mathcal{L}(\Gamma,\omega)$ = the weak operator closed algebra generated by $\{\lambda_{\omega}(\gamma)|\gamma\in \Gamma\}$; 
    \item the {\it $\omega$-twisted right group von Neumann algebra} $\mathcal{R}(\Gamma,\omega)$ = the weak operator closed algebra generated by $\{\rho_{\omega}(\gamma)|\gamma\in \Gamma\}$. 
\end{enumerate}
It is known that $\mathcal{R}(\Gamma,\overline{\omega})$ is the commutant of $\mathcal{L}(\Gamma,\omega)$ on $l^2(\Gamma)$, where $\overline{\omega}$ denotes the complex conjugate of $\omega$ (see \cite[\S 1]{Kleppner62}). 
If $\omega$ is trivial, $\mathcal{L}(\Gamma,\omega)$ reduces to $\mathcal{L}(\Gamma)$. 
Thus, if $H^2(\Gamma;\mathbb{T})$ is trivial, all these $\mathcal{L}(\Gamma,\omega)$ are isomorphic to the untwisted group von Neumann algebra $\mathcal{L}(\Gamma)$. 
For instance, as $\PSL(2,\mathbb{Z})\cong \mathbb{Z}/2\mathbb{Z}*\mathbb{Z}/3\mathbb{Z}$, we have $H^2(\PSL(2,\mathbb{Z});\mathbb{T})=1$ (see \cite[Corollary 6.2.10]{Web}).

There is a natural trace $\tau\colon \mathcal{L}(\Gamma,\omega)\to \mathbb{C}$  given by
\begin{center}
    $\tau(x)=\langle x\delta_e,\delta_e\rangle_{l^2(\Gamma)}$.
\end{center}
It gives an inner product on $\mathcal{L}(\Gamma,\omega)$ defined by $\langle x,y \rangle_{\tau}=\tau(xy^*)$ for $x,y\in \mathcal{L}(\Gamma,\omega)$. 
%The completion of $\mathcal{L}(\Gamma,\omega)$ with respect to the inner product $\langle -,- \rangle_{\tau}$ is exactly $l^2(\Gamma)$. 

Generally, for a tracial von Neumann algebra $M$ with a tracial state $\tau$,  
the GNS representation of $M$ gives us a Hilbert space $L^2(M)$ from the completion with respect to the inner product $\langle x,y\rangle_{\tau}=\tau(xy^*)$. 
One can show that $L^2(M)$ is exactly $l^2(\Gamma)$ when $M$ is the (twisted)  left or right von Neumann algebra of $\Gamma$. 

Suppose $\pi\colon M\to B(H)$ is a normal unital representation of $M$ with both $M$ and $H$ separable.  
There exists an isometry $u\colon H\to L^2(M)\otimes l^2(\mathbb{N})$, which commutes with the actions of $M$:
\begin{center}
$u\circ\pi(x)=(\lambda(x)\otimes\id_{l^2(\mathbb{N})} )\circ u$, $\forall x\in M$,
\end{center}
where $\lambda\colon M\to B(L^2(M))$ denotes the left multiplication.  
Then $p=uu^*$ is a projection in $B(L^2(M)\otimes l^2(\mathbb{N}))$ such that $H\cong p( L^2(M)\otimes l^2(\mathbb{N}))$ as modules over $M$. 
The following result is well-known (see \cite[Proposition 8.2.3]{APintrII1}). 

\begin{lemma}\label{ltrdim}
The correspondence $H\mapsto p$ above defines a bijection between the set of equivalence classes of left $M$-modules and the set of equivalence classes of projections in $(M'\cap B(L^2(M)))\otimes B(l^2(\mathbb{N}))$. 
\end{lemma}

The {\it von Neumann dimension} of the $M$-module $H$ is defined as
\begin{equation}
    \dim_M(H)=(\tau\otimes \Tr)(p)\in [0,\infty],
\end{equation}
which is independent of the choice of the particular projection $p$ in its equivalence class. 
We know that $\dim_M(L^2(M))=1$. 
If $M$ is a finite factor, i.e., $Z(M)\cong\mathbb{C}$, the tracial state is unique and $\dim_M(H)=\dim_M(H')$ if and only if $H$ and $H'$ are isomorphic as $M$-modules. 
When $M$ is not a factor, there is a $Z(M)$-valued trace which determines the isomorphism class of $M$-modules (see \cite{Bek04}).  

It is well-known that the dimensions are countably summable, i.e.,
\begin{center}
    $\dim_M(\oplus_i H_i)=\sum_i \dim_M(H_i)$.
\end{center} 
We will generalize this to direct integrals by introducing the following notion. 

Let $(M,\tau)$ be  a tracial von Neumann algebra. 
\begin{definition}[von Neumann densities]
Let $(X,\nu)$ be a measure space and $\{H_x\}_{x\in X}$ be a field of Hilbert spaces over $X$ such that there exists a constant $C>0$ and for any measurable $Y\subset X$, 
\begin{equation} 
H_Y=\int_{Y}^{\oplus }H_{x}d\nu(x) \text{~is an~} M\text{-module}  \text{~such that~}
    \dim_{M}H_Y=C\cdot \nu(Y). 
\end{equation}
We call $C\cdot d\nu(x)$ the \textbf{von Neumann density} of $H_x$ over $M$ and denote it by $d_{M}(H_x)$. 

If $M=\mathcal{L}(\Gamma)$ or $\mathcal{L}(\Gamma,\omega)$ for a contable discrete group $\Gamma$ and a $2$-cocycle $\omega$, we simply denote $d_{M}(H_x)$ by $ d_{\Gamma}(H_x)$ or $d_{\Gamma,\omega}(H_x)$ and call it \textbf{$\Gamma$-density} or \textbf{$(\Gamma,\omega)$-density} of $H_x$.    
\end{definition}

Here we do not assume locally that $H_{x}$ is a $M$-module. 
But if we assume $H_x$ is a $M$-module and $\nu(\{x\})=1$, we obtain $d_{M}(H_x)=\dim_M H_x$. 
Furthermore, we will show that
\begin{center}
$\dim_{M}\int_{Y}^{\oplus}H_{x}d\nu(x)=\int_{Y}d_{M}H_{x}d\nu(x)$ 
\end{center}
for any finite measure subset $Y$ of $X$ (see Proposition \ref{pvndimint}) 

Recall that the commutant $M'$ of the tracial von Neumann algebra $(M,\tau)$ on $L^2(M)$ is $JMJ$, where $J\colon L^2(M)\to L^2(M)$ is the conjugate linear map which extends $x\mapsto x^{*}$. 
The canonical trace on this commutant is given as $\tr_{M'}(JxJ)=\tau(x)$ for $x\in M$. 
The commutant of $M$ acting on $L^2(M)\otimes l^2(\mathbb{N})$ is $JMJ\otimes B(l^2(\mathbb{N}))$.
Then the trace on the commutant $JMJ\otimes B(l^2(\mathbb{N}))$ is given as 
\begin{center}
    $\Tr_{M'}(x)=(\tr_{M'}\otimes\Tr)(x)$,
\end{center}
where $\Tr$ is the canonical trace on $B(l^2(\mathbb{N}))$ that sends one-rank  projections to $1$.
\begin{proposition}\label{pvndimint}
Let $(X,\nu)$ be a separable finite measure space. 
Let $\{H_x|x\in X\}$ be a measurable field of modules over a finite tracial von Neumann algebra $(M,\tau)$. 
Suppose $\dim_{M}H_{x}$ is finite for each $x\in X$. 
We have
\begin{equation}\label{eintegvndim}
\dim_{M}\int_{X}^{\oplus}H_{x}d\nu(x)=\int_{X}\dim_{M}H_{x}d\nu(x)
\end{equation}
if both sides are finite. 
\end{proposition}
\begin{proof}
For each $x\in X$, we let $L^2(\mathbb{N})_{x}= L^2(\mathbb{N})$. 
There exists an $M$-linear isometric embedding
\begin{center}
    $u_x\colon H_x\to L^2(M)\otimes l^2(\mathbb{N})_{x}$
\end{center}
such that
\begin{center}
    $\dim_{M}H_x=\Tr_{M'}(u_{x}u_{x}^*)=(\tr_{M'}\otimes\Tr)(u_{x}u_{x}^*)$.  
\end{center}
Let $B_x=B(l^2(\mathbb{N})_x)_{\HS}$ be the subspace of Hilbert-Schmidt operators in $B(l^2(\mathbb{N})_x)$. 
It is known that $B_x$ is a Hilbert algebra equipped with the product $\langle a,b\rangle=\Tr(ab^{*})$ (see \cite[Appendix A.54]{DiCalg} and \cite[Chapter I.5]{DivNalg}). 
We let 
\begin{center}
    $A_x:=JMJ\otimes B(l^2(\mathbb{N})_x)_{\HS}=JMJ\otimes B_{x}$,
\end{center}
which is a Hilbert algebra with the inner product given by $\langle a,b\rangle=(\tr_{M'}\otimes \Tr)(ab^{*})$.  
As $\dim_{M}H_{x}$ is finite, we have $u_{x}u_{x}^{*}\in A_x$.

\iffalse
Recall the assumption that $\dim_{M}H_{x}$ is finite for all $x\in X$.

Thus we may take sequences  $\{a_{k}(x)\}_{k\geq 1}$ in $JMJ$ and $\{b_{k}(x)\}$ in $B_x$ such that 
\begin{equation}\label{euxaxbx}
    \lim_{N\to\infty}\sum_{k=1}^{N}a_{k}(x)\otimes b_{k}(x)=u_{x}u_{x}^*
\end{equation}
with respect to the topology given by the inner product. 
\fi

Let us consider the map
\begin{equation}\label{euX}
\begin{aligned}
    u_X=\int_{X}^{\oplus}u_x d\nu(x)\colon &\int_{X}^{\oplus}H_x d\nu(x)&\to &\int_{X}^{\oplus}L^2(M)\otimes l^2(\mathbb{N})_{x}d\nu(x),\\
    & \int_{X}^{\oplus}v_{x}d\nu(x) &\mapsto &\int_{X}^{\oplus}u_{x}(v_x)d\nu(x),
\end{aligned}
\end{equation}
which is also an $M$-linear isometric embedding. 
For the right side of \ref{euX}, we have
\begin{center}
    $\int_{X}^{\oplus}L^2(M)\otimes l^2(\mathbb{N})_{x}d\nu(x)\cong L^2(M)\otimes L^2(X,\nu)\otimes l^2(\mathbb{N})$. 
\end{center}

Note that $ \int_{X}^{\oplus}JMJ\otimes B(l^2(\mathbb{N})_{x})_{\HS}d\nu(x)$ is a Hilbert algebra whose inner product comes from the direct integral, which can be written as
\begin{center}
$\langle  \int_{X}^{\oplus} a_{x}\otimes v_{x}d\nu(x),\int_{X}^{\oplus} b_{x}\otimes w_{x}d\nu(x)\rangle=\int_{X}\tr_{M'}(a_{x}b_{x}^{*})\cdot\Tr_{L^2(\mathbb{N})}(v_{x}w_{x}^{*})d\nu(x)$,
\end{center}
for $\int_{X}^{\oplus} a_{x}\otimes v_{x}d\nu(x),\int_{X}^{\oplus} b_{x}\otimes w_{x}d\nu(x) \in  \int_{X}^{\oplus}JMJ\otimes B(l^2(\mathbb{N})_{x})_{\HS}d\nu(x)$. 
By \cite[Chapter II.5 Theorem 1]{DivNalg} and the inner product described above, its natural trace is given as
\begin{center}
$\int_{X}^{\oplus}\tr_{M'}\otimes\Tr_{L^2(\mathbb{N})}d\nu(x)$. 
\end{center}
By the assumption that both sides of \ref{eintegvndim} are finite, 
we have
\begin{center}
    $u_{X}u_{X}^*\in  \int_{X}^{\oplus}JMJ\otimes B(l^2(\mathbb{N})_{x})_{\HS}d\nu(x)$, 
\end{center} 
whose norm is finite. 

Hence, by the definition of $M$-dimensions and the assumption on their finiteness, we have
\begin{equation}
\begin{aligned}
    \dim_{M}\int_{X}^{\oplus}H_x d\nu(x)
&=\left(\int_{X}^{\oplus}\tr_{M'}\otimes\Tr_{L^2(\mathbb{N})}d\nu(x)\right)(u_{X}u_{X}^{*})\\
&=\int_{X}\left(\tr_{M'}\otimes\Tr_{L^2(\mathbb{N})}\right)(u_{x}u_{x}^*)d\nu(x)\\
&=\int_{X}\dim_{M}H_{x}d\nu(x). 
\end{aligned}
\end{equation}
\iffalse
\begin{equation}
\begin{aligned}
    \dim_{M}\int_{X}^{\oplus}H_x d\nu(x)&=\left(\tr_{M'}\otimes \Tr_{L^2(X,\nu)\otimes l^2(\mathbb{N})}\right)(u_{X}u_{X}^{*})\\
    &=\left(\tr_{M'}\otimes \Tr_{L^2(X,\nu)}\otimes \Tr_{L^2(\mathbb{N})}\right)(u_{X}u_{X}^{*})\\
    &=\left(\tr_{M'}\otimes \Tr_{L^2(X,\nu)}\otimes \Tr_{L^2(\mathbb{N})}\right)\left(\int_{X}^{\oplus}\lim_{N\to \infty}\sum_{k=1}^{N}a_{k}(x)\otimes b_{k}(x) d\nu(x)\right)\\
    &=\lim_{N\to \infty} \sum_{k=1}^{N} \int_{X}\left(\tr_{M'}(a_{k}(x))\cdot \Tr_{L^2(\mathbb{N})}(b_{k}(x))\right)d\nu(x)\\
&=\int_{X}\left(\tr_{M'}\otimes\Tr_{L^2(\mathbb{N})}\right)(u_{x}u_{x}^*)d\nu(x)\\
&=\int_{X}\dim_{M}H_{x}d\nu(x). 
\end{aligned}
\end{equation}
\fi
\end{proof}

\section{The Atiyah-Schmid formula for projective representations}\label{sASproj}

We let $G$ be a unimodular locally compact group of type I. 
Following \cite[\S 7.2]{Fo2}, a group is called type I if each primary representation\footnote{A unitary representation $(\pi,H)$ of $G$ is called primary if $\pi(G)''$, the von Neumann algebra it generates, is a factor, i.e., $Z(\pi(G)'')\cong \mathbb{C}$. 
Equivalently, assuming $(\pi,H)$ is a direct sum of irreducible representations, $(\pi,H)$ is primary if and only if $(\pi,H)$ is a direct sum of some single irreducible representation}  
generates a type I factor. 
More precisely, for any unitary representation $(\pi,H)$ of $G$, if $\pi(G)''$ is factor, then $\pi(G)''$ is a factor of type I, i.e. $\pi(G)''\cong B(K)$ for some Hilbert space $K$ (possibly infinite-dimensional).
The class of type I groups contains real linear algebraic groups (see \cite[\S 8.4]{Kiri76}), 
reductive $p$-adic groups (see \cite{Bnst74}), and
also reductive adelic group (see \cite[Appendix]{Clz07}).

Let $\omega$ be a normalized {\it 2-cocycle} of $G$, i.e.,  a Borel map $\omega\colon G\times G\to \mathbb{T}$.
%with $\mathbb{T}=\{z\in\mathbb{C}| |z|=1 \}$ such that
\begin{center}
$\omega(g,h)\omega(gh,k)=\omega(g,hk)\omega(h,k)$ and $\omega(g,e)=\omega(e,g)=1$
\end{center}
for all $g,h,k\in G$. 
Let $Z^2(G,\mathbb{T})$ be the group of normalized 2-cocycles of $G$. 
Two 2-cocycles $\omega_1,\omega_2\in Z^2(G,\mathbb{T})$ are usually called {\it cohomologous} if there exists $\varphi\colon G\to \mathbb{T}$ such that $\varphi(e)=1$ and $\omega_1(g,h)\overline{\omega_2}(g,h)=\varphi(gh)\overline{\varphi}(g)\overline{\varphi}(h)$ for all $g,h\in G$.  
Then $H^2(G,\mathbb{T})$ is defined as the quotient of $Z^2(G,\mathbb{T})$ by the abelian group generated by the 2-cocycles which are cohomologus to $1$. 

Given a 2-cocycle $\omega$, by a $\omega$-projective representation we mean a continuous map $\pi\colon G\to U(H_{\pi})$ such that $\pi(g)\pi(h)=\omega(g,h)\pi(gh)$ for all $g,h\in G$. 
We let
\begin{itemize}
\item $\Pi(G,\omega)$= the set of equivalence classes of $\omega$-projective irreducible representations of $G$. 
    \item $\lambda_{\omega}$= the $\omega$-projective left regular representation of $G$ on $L^2(G)$ defined as
    \begin{equation}\label{eprojleftreg}
        \lambda_{\omega}(g)f(x)=\omega(x^{-1},g)f(g^{-1}x)
    \end{equation}
for all $g,x\in G$, $f\in L^2(G)$.
\item $\rho_{\omega}$= the $\omega$-projective left and right regular representation of $G$ on $L^2(G)$ defined as
\begin{equation}\label{eprojrightreg}
    \rho_{\omega}(g)f(x)=\omega(g^{-1},x^{-1})f(xg)
\end{equation}
for all $g,x\in G$, $f\in L^2(G)$.
\end{itemize}
The following result % is an altered statement of a theorem that 
was proved by Kleppner and Lipsman (see \cite[I.Theorem 7.1]{KlepLips1972}). 

\begin{theorem}\label{tprojPthm}[Kleppner-Lipsman, 1972]
Let $G$ be a locally compact unimodular group with Haar measure $\mu$.  
There exists a positive standard Borel measure $\nu_{G,\omega}$ on $\Pi(G,\omega)$ and % a measurable field $\pi\mapsto H_{\pi}\otimes H_{\pi}^{*}$ and 
a measurable field of representations $(\pi,H_{\pi})$ such that
\begin{enumerate}
    %\item $\pi_{\zeta}\in \zeta$ for $\nu_{\overline{G},\sigma}$-almost all $\zeta$;
    \item there exists an isomorphism $\Psi\colon L^2(G,\mu)\to \int_{\Pi(G,\omega)}^{\oplus}H_{\pi}\otimes H_{\pi}^{*}\dd \nu_{G,\omega}(\pi)$ given by the extension of the Fourier transform $\mathcal{F}\colon f\mapsto \widehat{f}(\pi)=\int_{G}f(g)\pi(g^{-1})d\mu (g)$ with $f\in L^1(G)$, 
    which intertwines
    \begin{enumerate}
        \item $\lambda_{\omega}$ with $\int_{\Pi(G,\omega)}^{\oplus}\pi\otimes \id_{H_\pi} \dd \nu_{G,\omega}(\pi)$;
        \item $\rho_{\omega}$ with $\int_{\Pi(G,\omega)}^{\oplus}\id_{H_\pi}\otimes \overline{\pi} \dd \nu_{G,\omega}(\pi)$;
    \end{enumerate}
    %\item $(\Psi f)(\pi)=\pi(f)$ for $f\in L^1(G)$;
    \item For $f,h\in \mathcal{J}^1=L^1(G)\cap L^2(G)$, we have
    \begin{center}
        $\int_{G}f(g)\overline{h}(g)\dd \mu_{G}(g)=\int_{\Pi(G,\omega)}\Tr(\pi(f)\pi(h)^{*})d \nu_{G,\omega}(\pi)$. 
    \end{center} 
\end{enumerate}
\end{theorem}

We will call $\nu_{G,\omega}$ the {\it $\omega$-Plancherel measure} on $\Pi(G,\omega)$. 
Note that if $\omega$ is trivial, this theorem reduces to the ordinary Plancherel theorem (see \cite[\S 7]{Fo2}). 

Let $X$ be a $\nu_{G,\omega}$-measurable subset of 
$\Pi(G,\omega)$ with finite $\omega$-Plancherel measure, i.e., $\nu_{G,\omega}(X)<\infty$. 
Define
\begin{center}
$H_X=\int_{X}^{\oplus}H_{\pi}d\nu_{G,\omega}(\pi)$,   
\end{center} 
which is the direct integral of the underlying Hilbert space $H_{\pi}$ of the $\omega$-projective representation $\pi\in X$.  
Suppose $\{e_k(\pi)\}_{k\geq 1}$ is an orthonormal basis of $H_{\pi}$. 
We have the following natural isometric isomorphism from $H_X$ to a subspace of $L^2(G)$: 
\begin{equation}\label{eHXsub}
\begin{aligned}
H_X~~~~~~~~~&\cong ~~~\int_{X}^{\oplus}H_{\pi}\otimes e_1(\pi)^* d\nu_{G,\omega}(\pi)\\
v=\int_{X}^{\oplus}v(\pi)d\nu(\pi)~~&\mapsto~~~ \int_{X}v(\pi)\otimes e_1(\pi)^{*}d\nu_{G,\omega}(\pi),
\end{aligned}
\end{equation}
which intertwines the following two $\omega$-actions: 
\begin{center}
    $\lambda_{\omega,X}=\int_{X}^{\oplus}\pi d\nu_{G,\omega}(\pi)$ and $\lambda_{\omega}$
\end{center}
of $G$ on $H_X$ and $L^2(G)$ respectively. 
Therefore we will not distinguish these two spaces and denote them both by $H_X$. 

The $(G,\omega)$-equivariant projecion $P_X\colon L^2(G)\to H_X$ can be defined on a dense subspace of $L^2(G)$ as follows:  
{\small
\begin{equation}\label{eHXsubspace}
\int_{\Pi(G,\omega)}^{\oplus}\left(\sum\limits_{i,j\geq 1}a_{i,j}(\pi)e_j(\pi)\otimes e_i(\pi)^*\right)d\nu_{G,\omega}(\pi)\mapsto \int_{X}^{\oplus}\left(\sum\limits_{j\geq 1}a_{1,j}(\pi)e_j(\pi)\otimes e_1(\pi)^*\right)d\nu_{G,\omega}(\pi)
,
\end{equation}
}
where all but finite $a_{i,j}(\pi)\in \mathbb{C}$ are zero for each $\pi$. 
%When $X$ is clear in the context, we will write $P$ for $P_X$. 

Given two vectors $v=\int_{X}^{\oplus}v(\pi)d\nu_{G,\omega}(\pi)$ and $w=\int_{X}^{\oplus}w(\pi)d\nu_{G,\omega}(\pi)$ in $H_X$ with $v(\pi),w(\pi)\in H_{\pi}$, we have $v(\pi)\otimes w(\pi)^{*}\in H_{\pi}\otimes H_{\pi}^*$.
As we can identify $H_{\pi}\otimes H_{\pi}^*$ with the space of Hilbert-Schmidt operator on $H_{\pi}$, 
we will also treat $v(\pi)\otimes w(\pi)^{*}$ as a Hilbert-Schmidt operator in $B(H_{\pi})$. 
We define a function on $G$ by
\begin{center}
$C_{v,w}(g)=\langle \lambda_{\omega ,X}(g^{-1})v,w\rangle_{H_X},$%$=\int_{X}\Tr(\pi(g)^{*}v(\pi)\otimes w(\pi)^*)d\nu(\pi)$. 
\end{center}
which is the matrix coefficient function attached to $v,w$. 

The twisted convolution $\lambda
_{\omega}\colon L^1(G)\to B(L^2(G))$ is given by 
\begin{equation}\label{etwinv}
\begin{aligned}
    (\lambda_{\omega}(f)h)(x)=(f*_{\omega}h)(x):=\int_{G}\omega(x,y^{-1})f(xy^{-1})h(y)d\mu(y).
\end{aligned}
\end{equation}
Let $\|f\|_1$ denote the $L^1$-norm of $f\in L^1(G)$. 
\begin{lemma}
    $\|\lambda_{\omega}(f)\|\leq \|f\|_1$.
\end{lemma}
\begin{proof}
By Minkowski's integral inequality, we have
\begin{equation*}
\begin{aligned}
   \|\lambda_{\omega}(f)h\|_2&=\left(\int_{G}\abs[\Big]{\int_{G}\omega(x,y^{-1})f(xy^{-1})h(y)d\mu(y)}^2 d\mu(x)\right)^{1/2}\\
   &=\left(\int_{G}\abs[\Big]{\int_{G}\omega(x,x^{-1}y)f(y)h(y^{-1}x)d\mu(y)}^2 d\mu(x)\right)^{1/2}\\
   &\leq \int_{G}|f(y)|\left(\int_{G}|\omega(x,x^{-1}y)h(y^{-1}x)|d\mu(x)\right)^{1/2}d\mu(y)\\
   &\leq \|f\|_{1}\cdot \|h\|_{2}.
\end{aligned}
\end{equation*}
\end{proof}

%Recall that $P_X$ is the $G$-equivariant projection $L^2(G)\to H_X$ and $\mathcal{J}^{1}=L^1(G)\cap L^2(G)$. 
%In the following discussion, we will not distinguish $v(\pi)$ with $v(\pi)\otimes e_1(\pi)^{*}$, both of which denote the Hilbert-Schmidt operator on $H_{\pi}$. 
%Then we have
%where $\langle \cdot,\cdot\rangle_{\pi}$ is the Hilbert-Schmidt norm on $H_{\pi}$.

\begin{lemma}\label{lmxcoef}
\begin{enumerate}
    \item For $v,w\in H_X$ with $v\in \mathcal{F}(\mathcal{J}^{1})$, we have 
    \begin{center}
        $C_{v,w}\in L^2(G)$.
    \end{center} 
    \item For $v_1,v_2,w_1,w_2\in H_X$ with $v_1,v_2\in \mathcal{F}(\mathcal{J}^{1})$, we have
    \begin{equation}\label{eLem3.2}
        \langle C_{v_1,w_1},C_{v_2,w_2}\rangle_{L^2(G)}=\int_{X}\Tr(v_1(\pi) w_1(\pi)^{*} w_2(\pi) v_2(\pi)^{*})d\nu_{G,\omega}(\pi). 
    \end{equation}
\end{enumerate}

\end{lemma}

\begin{proof}
For simplicity, we will write $*$ for the $\omega$-convolution $*_{\omega}$ given in \ref{etwinv} and write $\lambda$ for $\lambda_{\omega}$.  

We let $f_v,f_w\in L^2(G)$ be the inverse image of $v,w$ under the Fourier transform, i.e., $\mathcal{F}(f_v)=v$ and $\mathcal{F}(f_w)=w$. 
As $f_v\in \mathcal{J}^1$ by assumption,  
we have
\begin{equation*}
\begin{aligned}
C_{v,w}(g)&=\langle \lambda_{\omega,X}(g^{-1})v,w\rangle_{H_X}=\langle\lambda_{\omega}(g^{-1})f_v,f_w\rangle_{L^2(G)}\\
%&=\int_{G}f_v(gx)\overline{f_{w}(x)}dx=\int_{G}f_v(gx)f_{w}^{*}(x^{-1})dx\\
&=\int_{G}\omega(x^{-1},g^{-1})f_v(gx)\overline{f_{w}(x)}dx=\int_{G}\omega(x^{-1},g^{-1})f_v(gx)f_{w}^{*}(x^{-1})dx\\
&=(f_{v}*f_{w}^{*})(g),
\end{aligned}
\end{equation*}
where $f^{*}(g)=\overline{f(g^{-1})}$.
This shows that $C_{v,w}(g)\in L^2(G)$. 

For the equality, we let $f_{v_1},f_{v_2},f_{w_1},f_{w_2}\in L^2(G)$ be the Fourier inverse image of $v_1,v_2,w_1,w_2\in H_X$ such that $f_{v_1},f_{v_2}\in \mathcal{J}^1$ by assumption. 
Let $\{f_{w_{1,j}}\}_{j\geq 1},\{f_{w_{2,k}}\}_{k\geq 1}$ be sequences in $\mathcal{J}^1$ such that $\lim_{j\to\infty}\|f_{w_{1,j}}-f_{w_1}\|_{2}=0$ and $\lim_{k\to\infty}\|f_{w_{2,k}}-f_{w_2}\|_{2}=0$. 
We also let $\{w_{1,j}\}_{j\geq 1},\{w_{2,k}\}_{k\geq 1}$ be the associated Fourier transformations of these functions.
Please note that $w_{1,j}(\pi),w_{2,k}(\pi)$ is a Hilbert-Schmidt operator for $\nu_{G,\omega}$-almost every $\pi$. 

We first observe that
\begin{equation}
\begin{aligned}
  \lim_{k\to \infty}\|C_{f_{v_i},f_{w_{i,k}}}-C_{f_{v_i},f_{w_{i}}}\|_{2}&=\lim_{k\to \infty}\|\lambda(f_{v_i})(f_{w_{i,k}}-f_{w_{i}})\|_{2}\\
  &\leq \lim_{k\to \infty}\|\lambda(f_{v_i})\|\cdot \|(f_{w_{i,k}}-f_{w_{i}})\|_{2}\\
  &\leq \lim_{k\to \infty}\|f_{v_i}\|_{1}\cdot \|(f_{w_{i,k}}-f_{w_{i}})\|_{2}=0,
\end{aligned}
\end{equation}
for $i=1,2$. 
Thus we obtain
\begin{equation}\label{eLem3.2-1}
\begin{aligned}
&\langle C_{v_1,w_1},C_{v_2,w_2}\rangle_{L^2(G)}\\=&\langle f_{v_1} * f_{w_1}^{*},f_{v_2} * f_{w_2}^{*}\rangle_{L^2(G)}=\lim\limits_{j,k\to\infty }\langle f_{v_1} * f_{w_{1,j}}^{*},f_{v_2} * f_{w_{2,k}}^{*}\rangle_{L^2(G)}\\
=&\lim\limits_{j,k\to\infty }\int_{\Pi(G,\omega)}\Tr\left(v_1(\pi)w_{1,j}^{*}(\pi)w_{2,k}(\pi)v_2^{*}(\pi)\right)d\nu_{G,\omega}(\pi)
\end{aligned}
\end{equation}
Since $w_{1,j}^{*}(\pi)w_{2,k}$ a trace class operator, $v_1(\pi)w_{1,j}^{*}(\pi)w_{2,j}(\pi)$ is also trace class. 
Thus $\Tr(v_1(\pi)w_{1,j}^{*}(\pi)w_{2,k}(\pi)v_2^{*}(\pi))=\Tr(v_2^{*}(\pi)v_1(\pi)w_{1,j}^{*}(\pi)w_{2,k}(\pi))$ and Equation \ref{eLem3.2-1} equals to
\begin{equation}\label{eLem3.2_2}
\begin{aligned}
\lim\limits_{j,k\to\infty }\int_{\Pi(G,\omega)}\Tr\left(v_2^{*}(\pi)v_1(\pi)w_{1,j}^{*}(\pi)w_{2,k}(\pi)\right)d\nu_{G,\omega}(\pi),
\end{aligned}
\end{equation}
which is the sum of the following three terms
\begin{enumerate}
    \item $ \lim\limits_{j,k\to\infty }\int_{\Pi(G,\omega)}\Tr\left(v_2^{*}(\pi)v_1(\pi)\left(w_{1,j}^{*}(\pi)w_{2,k}(\pi)-w_{1,j}^{*}(\pi)w_{2}(\pi)\right) \right)d\nu_{G,\omega}(\pi)$;
    \item $ \lim\limits_{j\to\infty }\int_{\Pi(G,\omega)}\Tr\left(v_2^{*}(\pi)v_1(\pi)\left(w_{1,j}^{*}(\pi)w_{2}(\pi)-w_{1}^{*}(\pi)w_{2}(\pi)\right) \right)d\nu_{G,\omega}(\pi)$;
    \item $ \int_{\Pi(G,\omega)}\Tr\left(v_2^{*}(\pi)v_1(\pi)w_{1}^{*}(\pi)w_{2}(\pi) \right)d\nu_{G,\omega}(\pi)$.
\end{enumerate}

Note that the last term above is exactly the right side of the desired equality since all $v_i,w_i$ have their support in $X$.  
It then suffices to show that the first two are trivial. 
For the first one, we have
\begin{equation}
\begin{aligned}
&\lim\limits_{j,k\to\infty }\int_{\Pi(G,\omega)}\Tr\left(v_2^{*}(\pi)v_1(\pi)\left(w_{1,j}^{*}(\pi)w_{2,k}(\pi)-w_{1,j}^{*}(\pi)w_{2}(\pi)\right) \right)d\nu_{G,\omega}(\pi)\\
=&\lim\limits_{j,k\to\infty }\int_{\Pi(G,\omega)}\Tr\left(v_2^{*}(\pi)v_1(\pi)w_{1,j}^{*}(\pi)\cdot \left( w_{2,k}(\pi)-w_{2}(\pi) \right) \right)d\nu_{G,\omega}(\pi)\\
=&\lim\limits_{j,k\to\infty }\langle f_{v_2}^{*}*f_{v_1}*f_{w_{1,j}},f_{w_{2,k}}-f_{w_2}\rangle_{L^2(G)}=0,
\end{aligned}
\end{equation}
which follows the fact that
\begin{equation*}
\begin{aligned}
    &\lim\limits_{j,k\to\infty }|\langle f_{v_2}^{*}*f_{v_1}*f_{w_{1,j}},f_{w_{2,k}}-f_{w_2}\rangle_{L^2(G)}|\\
    \leq &\|f_{v_1}\|_{1}\cdot \|f_{v_2}\|_{1}\cdot \lim\limits_{j\to\infty }\|f_{w_{1,j}}\|_{2}\cdot \lim\limits_{k\to\infty }\|f_{w_{2,k}}-f_{w_2}\|_{2}=0. 
\end{aligned}
\end{equation*}

For the second term, we let $h\in L^2(G)$ such that its Fourier transform at each $\pi$ is $w_{2}(\pi)v_2^{*}(\pi)v_1(\pi)$, i.e., $\mathcal{F}(h)=w_2 v_2^{*}v_1$. 
Then we have
\begin{equation}
\begin{aligned}
&\lim\limits_{j\to\infty }\int_{\Pi(G,\omega)}\Tr\left(v_2^{*}(\pi)v_1(\pi)\left(w_{1,j}^{*}(\pi)w_{2}(\pi)-w_{1}^{*}(\pi)w_{2}(\pi)\right) \right)d\nu_{G,\omega}(\pi)\\
    =&\lim\limits_{j\to\infty }\int_{\Pi(G,\omega)}\Tr\left(w_{2}(\pi)v_2^{*}(\pi)v_1(\pi)\left(w_{1,j}^{*}(\pi)-w_{1}^{*}(\pi)\right) \right)d\nu_{G,\omega}(\pi)\\
    =&\lim\limits_{j\to\infty }\langle h,f_{w_{1,j}}-f_{w_{1}}\rangle_{L^2(G)}=0,
\end{aligned}
\end{equation}
by the assumption that $\lim_{j\to\infty}\|f_{w_{1,j}}-f_{w_1}\|_{2}=0$. 
\end{proof}

Now we let $\Gamma$ be a discrete subgroup of $G$ which a lattice, i.e., $\mu(\Gamma/G)<\infty$.
The measure $\mu(\Gamma/G)$ is called {\it covolume} of $\Gamma$ \footnote{Note the covolume depends on the Haar measure $\mu$.}. 
Let $D\subset G$ be a {\it fundamental domain} for $\Gamma$, i.e., $\mu(G\backslash \cup_{\gamma\in\Gamma}\gamma D)=0$ and 
$\mu(\gamma_1 D\cap \gamma_2 D)=0$ if $\gamma_1\neq \gamma_2$ in $\Gamma$.

%: if we take another Haar measure $\mu'$ such that $\mu'=k\cdot \mu$ for some $k>0$, then $\mu'(D)=k\cdot \mu(D)$. 

There is a natural isomorphism $L^2(G)\cong l^2(\Gamma)\otimes L^2(D,\mu)$ given by
\begin{center}
$\phi\mapsto \sum_{\gamma\in\Gamma}\delta_{\gamma}\otimes \phi_{\gamma}$ with $\phi_{\gamma}(z)=\phi(\gamma\cdot z)$,
\end{center}
where $z\in D$ and $\gamma\in \Gamma$.
Let $\lambda_{\omega,G}(\gamma)$ denotes the $\omega$-projective representation of $\Gamma$ on $L^2(G)$. 
We can show that
\begin{center}
    $\lambda_{\omega,G}(\gamma)=\lambda_{\omega}(\gamma)\otimes \id_{L^2(D)}$
\end{center}
with respect to this decomposition. 

Let $\{d_k\}$ be an orthonormal basis of $L^2(D)$. 
\begin{lemma}\label{ldimsumnorm}
With the assumption above, we have
\begin{center}
    $\dim_{\mathcal{L}(\Gamma,\omega)}(H_X)=\sum_{k\geq 1}\|Pd_k\|_{H_X}^2$.
\end{center}
\end{lemma}
\begin{proof}
Let $u$ be the inclusion $H_X\to L^2(G)$ and $M=\mathcal{L}(\Gamma,\omega)$.  
We have $u^*u={\rm id}_{H_X}$ and $uu^*=P_X$. 
Note $L^2(G)\cong L^2(M)\otimes L^2(D,dg)$, where $L^2(M)\cong l^2(\Gamma)$ as an $M$-module and $L^2(D,dg)$ is regarded as a trivial $M$-module. 
Thus, by definition (see Lemma \ref{ltrdim}), we know
\begin{center}
$\dim_M(H_X)=\Tr_{M'\cap B(L^2(G))}(P_X)$, 
\end{center}
where $M'\cap B(L^2(G))=\{T\in B(L^2(G))|Tx=xT,~\forall x\in M\}$, the commutant of $M$ on $L^2(G)$.
On the right-hand side, 
\begin{center}
 $\Tr_{M'\cap B(L^2(G))}=\tr_{M'\cap B(L^2(M))}\otimes \Tr_{L^2(D)}$    
\end{center}
is the natural trace on $M'$. 

The commutant $M'$ is generated by the finite sums of the form
\begin{center}
$x=\sum_{\gamma\in \Gamma}\rho_{\overline{\omega}}(\gamma)\otimes a_{\gamma}$,
\end{center}
where $\rho_{\overline{\omega}}(\gamma)=J\lambda_{\omega}(\gamma)J\in M'\cap L^2(M)$ (here $J\colon L^2(M)\to L^2(M)$ is the conjugate linear isometry extended from $x\mapsto x^*$) and $a_{\gamma}$ is a finite rank operator in $B(L^2(D))$. 

Let $d_m^*\otimes d_n$ denotes the operator $\xi \mapsto \langle d_m,\xi\rangle\cdot d_n$ on $L^2(D)$. 
Then each $a_{\gamma}$ can be written as $a_{\gamma}=\sum_{m,n\geq 1}a_{\gamma,m,n}d_m^*\otimes d_n$ with $a_{\gamma,m,n}\in \mathbb{C}$ and all but finite many terms of $a_{m,n}$ are trivial. 
Thus we obtain
\begin{center}
$\Tr_{M'}(\rho_{\overline{\omega}}(\gamma)\otimes a_{\gamma})=\tr_M(\lambda_{\omega}(\gamma))\sum_{m\geq 1}a_{\gamma,m,m}=\delta_{e}(\lambda)\Tr_{L^2(D)}(a_\gamma)$.
\end{center}
This is equivalent to say
\begin{center}
$\Tr_{M'}(x)=\Tr_{L^2(D)}(a_e)$. 
\end{center}

Let $Q$ be the projection of $L^2(G)$ onto $L^2(D)\cong \mathbb{C}\delta_e\otimes L^2(D)$. 
Then $\Tr_{L^2(D)}(y)=\Tr_{L^2(G)}(QyQ)$ for $y\in B(L^2(D))$. 
We have
\begin{equation}\label{eTrQxQ}
    \Tr_{M'}(x)=\Tr_{L^2(D)}(a_e)=\Tr_{L^2(G)}(Qa_eQ)=\Tr_{L^2(G)}(QxQ)
\end{equation}
As $P_X$ is a strong limit of elements that have the same form as $x$ above and the traces are normal, the formula (\ref{eTrQxQ}) holds for $x=P_X$. 
Thus we obtain
\begin{equation*}
\begin{aligned}
\dim_M(H_X)&=\Tr_{M'}(P)=\Tr_{L^2(G,dg)}(QPQ)\\
&=\sum\nolimits_{k\geq 1}\langle QPQd_k,d_k\rangle_{L^2(G)}=\sum\nolimits_{k\geq 1}\langle Qd_k,PQd_k\rangle_{L^2(G)}\\
&=\sum\nolimits_{k\geq 1}\langle d_k,Pd_k\rangle_{L^2(G)}=\sum\nolimits_{k\geq 1}\langle Pd_k,Pd_k\rangle_{L^2(G)}\\
&=\sum\nolimits_{k\geq 1}\langle Pd_k,Pd_k\rangle_{H_X}=\sum\nolimits_{k\geq 1}\|Pd_k\|_{H_X}^2
\end{aligned}
\end{equation*}
\end{proof}

Let $\omega$ be a 2-cocycle of $G$ and $\nu_{G,\omega}$ be the Plancherel measure on $\Pi(G,\omega)$, the $\omega$-projective irreducible representations of $G$ (see Theorem \ref{tprojPthm}). 
Recall that for $X\subset\Pi(G,\omega)$ such that $\nu_{G,\omega}(X)<\infty$,  
\begin{center}
    $H_X=\int_X^{\oplus} H_{\pi}d\nu_{G,\omega}(\pi)$.
\end{center}
\begin{theorem}\label{tdimmeas}
Let $\Gamma$ be a lattice of $G$.   
We have 
\begin{equation}\label{eASf1}
   \dim_{\mathcal{L}(\Gamma,\omega)}H_X=\mu(\Gamma/G)\cdot \nu_{G,\omega}(X),
\end{equation}
or equivalently, 
\begin{equation}\label{eASf1a}
  d_{\Gamma,\omega}(H_{\pi})=\mu(\Gamma/G)\cdot d\nu_{G,\omega}(\pi),  
\end{equation}
\end{theorem}

\begin{proof}
We take a vector $\eta=\int_{X}^{\oplus}\eta(\pi)d\nu_{G,\omega}(\pi)$ in $H_X$ such that $\|\eta(\pi)\|_{H_{\pi}}^2=\frac{1}{\nu_{G,\omega}(X)}$ almost everywhere in $X$. 
Then $\eta$ is a unit vector in $H_X$ and also in $L^2(G)$. 
%For each $\pi\in X$, we take a vector $\eta(\pi)\in H_{\pi}$ such that $\|\eta(\pi)\|_{H_{\pi}}^2=\frac{1}{\nu(X)}$. Considering the vector $\eta=\int_{X}\eta(\pi)d\nu(\pi)$, 
%one has $\|\eta\|_{H_X}^2=\int_X\|\eta(\pi)\|_{H_{\pi}}^2d\nu(\pi)=1$, i.e., $\eta$ is a unit vector in $H_X$ (and also in $L^2(G)$). 

As $\mu(D)<\infty$, we have $L^2(D)\subset L^1(D)$.
We may take the basis $\{d_k\}_{k\geq 1}$ from the functions in $\mathcal{J}^1=L^1(G)\cap L^2(G)$, whose supports are contained in $D$. 

Observe $\{\delta_{\gamma}\otimes d_k\}_{\gamma\in \Gamma,n\geq 1}$ is an orthogonal basis of $L^2(G,\mu)$ via the isomorphism $L^2(G)\cong l^2(\Gamma)\otimes L^2(D,\mu)$.  
We identify $\delta_{\gamma}\otimes d_k$ with $\rho(\gamma)d_k$ and $\lambda(\gamma)^{-1}d_k$ for $k\geq 1$ and $\gamma\in \Gamma$. 
Please note that $\{\lambda_{\omega}(\gamma)^{-1}d_k|\gamma\in \Gamma,k\geq 1\}$ gives an orthonormal basis of $L^2(G)$. 
Hence, for each $g\in G$, we have
\begin{center}
$1=\|\lambda_{\omega,X}(g)\eta\|_{H_X}^2=\|\lambda_{\omega}(g)\eta\|_{L^2(G)}^2=\sum_{\gamma\in \Gamma,k\geq 1}|\langle \lambda_{\omega}(g)\eta,\lambda_{\omega}(\gamma)d_k\rangle_{L^2(G)}|^2$. 
\end{center}
Consequently, as $P_X$ commutes with the $G$-actions, we obtain: 
\begin{equation*}
\begin{aligned}
{\rm covol}(\Gamma)&=\int_{D}1d\mu(g)=\int_{D}\sum_{\gamma\in \Gamma,k\geq 1}|\langle \lambda_{\omega}(\gamma)^{*}\lambda_{\omega}(g)^{*}\eta,d_k\rangle|^2d\mu(g)\\
&=\sum_{k\geq 1}\int_{G}|\langle P\lambda_{\omega}(g)^{*}\eta,d_k\rangle_{L^2(G)}|^2d\mu(g)=\sum_{k\geq 1}\int_{G}|\langle \lambda_{\omega}(g)^{*}\eta,Pd_k\rangle_{H_X}|^2d\mu(g)\\
&=\sum_{k\geq 1}\int_{G}|\langle \lambda_{\omega}(g)^*Pd_k,\eta\rangle|^2d\mu(g)=\sum_{k\geq 1}\langle C_{Pd_k,\eta},C_{Pd_k,\eta}\rangle_{L^2(G)}\\
&=\sum_{k\geq 1}\int_X \Tr( (Pd_k)(\pi)\otimes \eta(\pi)^*\cdot (\eta(\pi)\otimes(Pd_k)(\pi)^*)d\nu_{G,\omega}(\pi)\\
&=\sum_{k\geq 1}\int_X \langle  (Pd_k)(\pi)\otimes \eta(\pi)^*,(Pd_n)(\pi)\otimes \eta(\pi)^*\rangle_{H_{\pi}\otimes H_{\pi}^*}d\nu_{G,\omega}(\pi)\\
&=\sum_{k\geq 1}\int_{X}\|\eta(\pi)\|_{H_\pi}^2\cdot \|(Pd_k)(\pi)\|_{H_{\pi}}^2d\nu_{G,\omega}(\pi)\\
&=\frac{1}{\nu_{G,\omega}(X)}\sum_{n\geq 1}\|Pd_k\|_{H_X}^2.
\end{aligned}
\end{equation*}
Here we may apply Lemma \ref{lmxcoef} in the third line above since all $d_k$ are functions in $\mathcal{J}^1$ with supports contained in $D$.  
This is $\dim_{\mathcal{L}(\Gamma,\omega)}(H_X)\cdot \nu(X)^{-1}$ by Lemma \ref{ldimsumnorm}. 
Hence we obtain $\dim_{\mathcal{L}(\Gamma,\omega)}H_X=\mu(\Gamma/G)\cdot \nu_{G,\omega}(X)$.  
\end{proof}

We should mention that the left side of Equation \ref{eASf1} is independent of the choice of the Haar measure $\mu$ in $G$: if $\mu'=c\cdot \mu$ is another Haar measure on $G$ for some $c>0$, the covolumes are related by  $\mu'(\Gamma/G)=c\cdot\mu(\Gamma/G)$ while $\nu'_{G,\omega}=c^{-1} \cdot\nu_{G,\omega}$ for the associated Plancherel measures. 
Thus the dependencies cancel out. 

\begin{remark}
Theorem \ref{tdimmeas} reduces the following special cases: 
\begin{enumerate}
\item  if $\omega$ is trivial and $X=\{\pi\}$ is a discrete series representation, it reduces to the original Atiyah-Schmid formula (see \cite[Theorem 3.3.2]{GHJ}). 
\item if $X=\{\pi\}$ is a discrete series representation, it reduces to \cite[Theorem 4.3]{Enstad22}. 
\item if $\omega$ is trivial, it reduces to the result in \cite[Theorem 4.1]{Y22} (see also a relevant approach by Peterson and Valette \cite{PetVlt14}). 
\end{enumerate}
\end{remark}

\section{The Atiyah-Schmid formula for reductive groups}\label{sASred}

Suppose $\mathbf{G}$ is a reductive group defined over $\mathbb{R}$ and $G=\mathbf{G}(\mathbb{R})$ is the real points. 
In general, the discrete group $\Gamma=\mathbf{G}(\mathbb{Z})$ is not a lattice of $G$, i.e., $\mu_G(\Gamma/G)=\infty$ (unless $G$ is semi-simple). 
We will give the Atiyah-Schmid formula for this case which generalizes the original one for semisimple Lie groups with their arithmetic subgroups.

We let $\mathbf{Z}$ be the center of $\mathbf{G}$ and $Z=\mathbf{Z}(\mathbb{R})$.
We let $\overline{G}=G/Z$, $\overline{\Gamma}=\Gamma/(Z\cap \Gamma)$ and $\widehat{G}$ be the unitary dual of $G$ which is equipped with the ordinary Plancherel measure $\nu_{G}$.  
  
\begin{theorem}\label{tASred}
Let $X\subset\widehat{G}$ such that $\nu_{G}(X)<\infty$ and $H_X=\int_X^{\oplus} H_{\pi}d\nu_{G}(\pi)$. 
We have
\begin{equation*}
    \dim_{\mathcal{L}(\mathbf{G}(\mathbb{Z}))}H_X=\frac{\mu_{\overline{G}}(\overline{\Gamma}/\overline{G})}{|Z\cap \Gamma|}\cdot \nu_{G}(X), 
\end{equation*}
or equivalently, $d_{\mathbf{G}(\mathbb{Z})}(H_{\pi})=\frac{\mu_{\overline{G}}(\overline{\Gamma}/\overline{G})}{|Z\cap \Gamma|}\cdot d\nu_{G}(\pi)$. 
\end{theorem}

We need a decomposition result of the ordinary Plancherel measure proved by Kleppner and Lipsman (see \cite[\S 8,10]{KlepLips1972}) for the proof of this theorem. 
We start with a general setting that $G$ is a locally compact unimodular type I group. 
Let $N$ be a central subgroup of $G$, i.e., $N\subset Z(G)$. 
We will apply the "Mackey machine" (see \cite{Mackey58} and \cite[\S 1]{Rsbg94}) to construct the irreducible representations of $G$ by the characters of $N$ and the projective irreducible representations of $G/N$. 
\begin{enumerate}
    \item For $\gamma\in \widehat{N}$, there is a projective representation $\gamma'$ of $G$ such that
    \begin{center}
        $\gamma'(gh)=\omega_{\gamma}(g,h)\gamma'(g)\gamma'(h)$
    \end{center}
    for a 2-cocycle $\omega_{\gamma}$ which is unique in $H^2(G/N,\mathbb{T})$. 
   It is known that $\gamma'$ extends $\gamma$: $\gamma'|_{N}=\gamma$ (see \cite[\S 1]{KlpLips1973}).
   \item Let $\sigma$ be a $\overline{\omega_{\gamma}}$-projective representation of $G/N$ and $\sigma'$ be the lift of $\sigma$ to $G$. 
   \item $\pi_{\gamma,\sigma}=\gamma'\otimes \sigma'$ is an ordinary irreducible representation of $G$. 
   It is known that each $\pi\in \widehat{G}$ is of such a form (see \cite[\S 1]{KlpLips1973}). 
\end{enumerate}
The Plancherel measure of $G$ can be described by the central extension of $N$ as follows. 
\begin{lemma}\label{lcenPlanc}
The left and right regular representations of $G$ can be decomposed as:
\begin{equation*}
\begin{aligned}
\lambda_G&=\int_{\widehat{N}}^{\oplus}\int_{\Pi(G/N,\overline{\omega_{\gamma}})}^{\oplus}\pi_{\gamma,\sigma}\otimes \id_{\pi_{\gamma,\sigma}^*}\dd\nu_{G/N,\overline{\omega_{\gamma}}}\dd \nu_{N}(\gamma),\\
\rho_G&=\int_{\widehat{N}}^{\oplus}\int_{\Pi(G/N,\overline{\omega_{\gamma}})}^{\oplus} \id_{\pi_{\gamma,\sigma}}\otimes \pi_{\gamma,\sigma}^{*} \dd\nu_{G/N,\overline{\omega_{\gamma}}}\dd \nu_{N}(\gamma)
\end{aligned}
\end{equation*}
where $\nu_{G/N,\overline{\omega_{\gamma}}}$ is the Plancherel measure on the $\overline{\omega_{\gamma}}$-projective dual $\Pi(G/N,\overline{\omega_{\gamma}})$. 
In particular, 
\begin{center}
    $d\nu_G(\pi_{\gamma,\sigma})=d\nu_{N}(\gamma)d\nu_{G/N,\overline{\omega_{\gamma}}}(\sigma)$. 
\end{center}
\end{lemma}
\begin{proof}
It follows \cite[Theorem 10.2]{KlepLips1972} for the special case $N\subset Z(G)$. 
\end{proof}

\begin{proposition}\label{pquovdim}
Let $\Gamma$ be a countable discrete group and $K$ be a finite normal subgroup of $\Gamma$. 
Let $\omega\in H^2(\Gamma/K,\mathbb{T})$ and $H$ be a module over $\mathcal{L}(\Gamma/K,\omega)$. 
Then $H$ is a module over $\mathcal{L}(\Gamma,\omega)$ such that
\begin{center}
    $\dim_{\mathcal{L}(\Gamma,\omega)}H=\frac{1}{|K|}\dim_{\mathcal{L}(\Gamma/K,\omega)}H$,
\end{center}
where $\mathcal{L}(\Gamma,\omega)$ is the twisted group von Neumann algebra associated with the lifting of the 2-cocycle of $\omega$ to $H^2(\Gamma,\mathbb{T})$. 
\end{proposition}
\begin{proof}
Assume $K=\{k_{i}\}_{1\leq i\leq m}$. 
Take $\{g_j\}_{j\geq 1}$ as a family of representatives for the coset $\Gamma/ K$. 
Then $\{\delta_{g_{j}K}\}_{j\geq 1}$ form a basis of $l^2(\Gamma/ K)$ and $\{\delta_{g_{j}k_{i}}\}_{j\geq 1,1\leq i\leq m}$ form a basis of $l^2(\Gamma)$. 
Consider the linear map $T\colon l^2(\Gamma/K)\to l^2(\Gamma)$ given by
\begin{center}
$T(\delta_{g_{j}K})=\frac{1}{\sqrt{|K|}}\sum_{1\leq i\leq m}\delta_{g_{j}k_{i}}$.
\end{center}
We can check that $T$ gives a $(\Gamma,\omega)$-equivariant isometry if $l^{2}(\Gamma/K)$ is equipped with the $(\Gamma,\omega)$-action which passes from the $\Gamma$-action to this quotient.  

Let $\tr(x)=\langle x\delta_{e} ,\delta_{e}\rangle$ denote the canonical tracial state on $\mathcal{L}(\Gamma,\omega)'\cap B(l^2(\Gamma))=\mathcal{R}(\Gamma,\overline{\omega})$. 
Thus we have
\begin{center}
    $\dim_{\mathcal{L}(\Gamma,\omega)}l^2(\Gamma/K)=\tr(TT^*)=\langle TT^* x\delta_{e} ,\delta_{e}\rangle=\frac{1}{|K|}$. 
\end{center}
Assume $H$ is a module over $\mathcal{L}(\Gamma/K,\omega)$ such that $\dim_{\mathcal{L}(\Gamma/K,\omega)}H=n+\alpha$ with $n\in \mathbb{N}$ and $0\leq \alpha<1$. 
We know that, as modules over $\mathcal{L}(\Gamma/K,\omega)$ and $\mathcal{L}(\Gamma,\omega)$,
\begin{center}
    $H\cong l^2(\Gamma/K)^{\oplus n}\oplus l^2(\Gamma/K)p$, 
\end{center}
for some $p\in \mathcal{R}(\Gamma/K,\overline{\omega})$ such that $\tr(p)=\alpha$. 
By \cite[Proposition 3.2.5(e)]{GHJ}, we have 
\begin{equation*}
    \dim_{\mathcal{L}(\Gamma,\omega)}l^2(\Gamma/K)p=\tr(p)\dim_{\mathcal{L}(\Gamma,\omega)}l^2(\Gamma/K)=\frac{\alpha}{|K|}.
\end{equation*}
Thus $\dim_{\mathcal{L}(\Gamma,\omega)}H=\frac{n+\alpha}{|K|}=\frac{1}{|K|}\dim_{\mathcal{L}(\Gamma/K,\omega)}H$. 
\end{proof}

Now we can prove the main theorem. 

\begin{proof}[Proof of Theorem \ref{tASred}]
We know that $\overline{G}=G/Z$ is a semi-simple real group and thus $\overline{\Gamma}=\overline{G}(\mathbb{Z})$ is a lattice of $\overline{G}$: $\mu_{\overline{G}}(\overline{\Gamma}/\overline{G})<\infty$. 
Moreover, $\mathbf{Z}(\mathbb{R})^0$  (the connected component) is a central torus such that $[\mathbf{Z}(\mathbb{R}):\mathbf{Z}(\mathbb{R})^0]$ is finite. 
Thus $\mathbf{Z}(\mathbb{R})^0\cong (\mathbb{R}^{\times})^k$ for some $k\in \mathbb{N}$ and $\mathbf{Z}(\mathbb{Z})^0\cong (\mathbb{Z}^{\times})^k$, which is finite. 
Hence $Z\cap \Gamma=\mathbf{Z}(\mathbb{Z})$ is a finite group. 

For each $\gamma\in \widehat{Z}$, we take $Y_{\gamma}\subset \Pi(\overline{G},\overline{\omega_{\gamma}})$ such that $\nu_{\overline{G},\overline{\omega_{\gamma}}}(Y_{\gamma})<\infty$. 
We let $H_{Y_{\gamma}}=\int_{Y_{\gamma}}^{\oplus}\sigma d\nu_{\overline{G},\overline{\omega_{\gamma}}}(\sigma)$. 
By Theorem \ref{tdimmeas}, $\dim_{\mathcal{L}(\overline{\Gamma},\overline{\omega_{\gamma}})}H_{Y_{\gamma}}=\mu_{\overline{G}}(\overline{\Gamma}/\overline{G})\cdot \nu_{\overline{G},\overline{\omega_{\gamma}}}(Y_{\gamma})$. 
By Proposition \ref{pquovdim}, we have
\begin{equation*}
    \dim_{\mathcal{L}(\Gamma,\overline{\omega_{\gamma}})}H_{Y_{\gamma}}=\frac{1}{|Z\cap \Gamma|}\mu_{\overline{G}}(\overline{\Gamma}/\overline{G})\cdot \nu_{\overline{G},\overline{\omega_{\gamma}}}(Y_{\gamma}),
\end{equation*}
where $\overline{\omega_{\gamma}}$ also denotes its lift from $\overline{\Gamma}$ to $\Gamma$. 

Consider the space $\gamma\otimes H_{Y_{\gamma}}$, which is $\gamma\otimes\int_{{Y_{\gamma}}}^{\oplus}\sigma d\nu_{\overline{G},\overline{\omega_{\gamma}}}(\sigma)=\int_{{Y_{\gamma}}}^{\oplus}\gamma\otimes\sigma d\nu_{\overline{G},\overline{\omega_{\gamma}}}(\sigma)$. 
As $\gamma$ is a $\omega$-projective representation of $G$, $\gamma\otimes\sigma$ is an ordinary representation of $G$ and also of $\Gamma$. 
Thus, by tensoring the $\omega_{\gamma}$-projective character $\gamma$ of $Z$,  $\gamma\otimes H_{Y_{\gamma}}$ comes to be a module over $\mathcal{L}(\Gamma)$, whose von Neumann dimension is given as
\begin{equation*}
  \dim_{\mathcal{L}(\Gamma)}(\gamma\otimes H_{Y_{\gamma}})=\frac{1}{|Z\cap \Gamma|}\mu_{\overline{G}}(\overline{\Gamma}/\overline{G})\cdot \nu_{\overline{G},\overline{\omega_{\gamma}}}(Y_{\gamma}). 
\end{equation*}
Let $W$ be a $\nu_{Z}$-measurable subset of $\widehat{Z}$ such that $\nu_{Z}(W)$ is finite. 
By Proposition \ref{pvndimint}, we have
\begin{equation}\label{erealred1}
\begin{aligned}
    \dim_{\mathcal{L}(\Gamma)}\left(\int_{W}\gamma\otimes H_{Y_{\gamma}}d\nu_{Z}(\gamma)\right)&=\int_{W}\dim_{\mathcal{L}(\Gamma)}(\gamma\otimes H_{Y_{\gamma}})d\nu_{Z}(\gamma)\\
    &=\frac{1}{|Z\cap \Gamma|}\mu_{\overline{G}}(\overline{\Gamma}/\overline{G})\cdot\int_{W} \nu_{\overline{G},\overline{\omega_{\gamma}}}(Y_{\gamma})d\nu_{Z}(\gamma).
\end{aligned} 
\end{equation}
For a measurable $X\subset \widehat{G}$ and $\gamma\in \widehat{Z}$, we let $X_{\gamma}$ be the $\gamma$-slice of $X$, i.e. 
\begin{center}
    $X_{\gamma}=\{\sigma\in \Pi(\overline{G},\overline{\omega_{\gamma}})|\pi_{\gamma,\sigma}\in X\}$. 
\end{center} 
By $d\nu_G(\pi_{\gamma,\sigma})=d\nu_{N}(\gamma)d\nu_{G/N,\overline{\omega_{\gamma}}}(\sigma)$ (see Lemma \ref{lcenPlanc}) and Equation \ref{erealred1}, we obtain
\begin{equation*}
\begin{aligned}
    \dim_{\mathcal{L}(\Gamma)}\left(\int_{X}^{\oplus}\pi d\nu_{G}(\pi)\right)&=\dim_{\mathcal{L}(\Gamma)}\left(\int_{X}^{\oplus}\pi_{\gamma,\sigma} d\nu_{G}(\pi_{\gamma,\sigma})\right)\\
    &=\dim_{\mathcal{L}(\Gamma)}\left(\int_{\widehat{Z}}^{\oplus}\gamma\otimes\left( \int_{X_{\gamma}}^{\oplus}\sigma d\nu_{G/Z,\overline{\omega_{\gamma}}}(\sigma)\right)d\nu_{Z}(\gamma)\right)\\
    &=\dim_{\mathcal{L}(\Gamma)}\left(\int_{\widehat{Z}}^{\oplus}\gamma\otimes H_{X_{\gamma}}d\nu_{Z}(\gamma)\right)\\
    &=\frac{1}{|Z\cap \Gamma|}\mu_{\overline{G}}(\overline{\Gamma}/\overline{G})\cdot\int_{\widehat{Z}} \nu_{\overline{G},\overline{\omega_{\gamma}}}(X_{\gamma})d\nu_{Z}(\gamma)\\
    &=\frac{1}{|Z\cap \Gamma|}\mu_{\overline{G}}(\overline{\Gamma}/\overline{G})\cdot\nu_{G}(X).
\end{aligned} 
\end{equation*}
\end{proof}

\begin{remark}
For the $S$-arithmetic subgroups in reductive groups, 
we sometimes should apply Theorem \ref{tdimmeas} to the adjoint group $G/Z(G)$ with its projective representations instead of Theorem \ref{tASred} for $G$ itself. 

Let $F$ be a number field and $\mathcal{O}$ be the integral ring of $F$. Let $F_v$ denote the local field at a place $v$ and $V_{\infty}$ be the set of infinite places of $F$. 
Then $G(\mathcal{O})$ is an arithmetic subgroup of $G_{\infty}=\prod_{v\in V_{\infty}}G(F_v)$. 
By Dirichlet's unit Theorem (see \cite[Theorem 7.4]{NeuANT}), the unit group of $\mathcal{O}$ is an abelian group with free rank $r+s-1$ where $r,2s$ denotes the number of real and complex embeddings of $F$ such that $[F:\mathbb{Q}]=r+2s$. 
In this case, $Z(\mathcal{O})$ may not be finite. 
Theorem \ref{tASred} only applies to the pair $G(\mathcal{O}_F)\subset G_{\infty}$ when $F$ is $\mathbb{Q}$ or an imaginary quadratic field.

For a finite set $S$ of places such that $S$ contains $V_{\infty}$,  
let $\mathcal{O}_S$ be the ring of $S$-integers. 
For the $S$-arithmetic group $G(\mathcal{O}_S)$ in $G_S=\prod_{v\in S}G(F_v)$, 
$Z(\mathcal{O}_S)$ has a free part if $S$ contains a finite place (see \cite[Theorem 5.12]{PlaVla94}). 
Thus Theorem \ref{tASred} does not apply to this case. 
\end{remark}

\bibliographystyle{abbrv}
%\bibliography{references_list}
\typeout{}
\bibliography{MyLibrary} % ref

\begin{thebibliography}{10}

\bibitem{APintrII1}
C.~Anantharaman and S.~Popa.
\newblock An introduction to $\text{II}_1$ factors.
\newblock {\em preprint}, 8, 2017.

\bibitem{AS77}
M.~Atiyah and W.~Schmid.
\newblock A geometric construction of the discrete series for semisimple {L}ie groups.
\newblock {\em Invent. Math.}, 42:1--62, 1977.

\bibitem{Aty76}
M.~F. Atiyah.
\newblock Elliptic operators, discrete groups and von {N}eumann algebras.
\newblock In {\em Colloque ``{A}nalyse et {T}opologie'' en l'{H}onneur de {H}enri {C}artan ({O}rsay, 1974)}, volume No. 32-33 of {\em Ast\'erisque}, pages 43--72. Soc. Math. France, Paris, 1976.

\bibitem{Bek04}
B.~Bekka.
\newblock Square integrable representations, von {N}eumann algebras and an application to {G}abor analysis.
\newblock {\em J. Fourier Anal. Appl.}, 10(4):325--349, 2004.

\bibitem{Bnst74}
I.~N. Bern\v{s}te\u{\i}n.
\newblock All reductive {$p$}-adic groups are of type {I}.
\newblock {\em Funkcional. Anal. i Prilo\v{z}en.}, 8(2):3--6, 1974.

\bibitem{Clz07}
L.~Clozel.
\newblock Spectral theory of automorphic forms.
\newblock In {\em Automorphic forms and applications}, volume~12 of {\em IAS/Park City Math. Ser.}, pages 43--93. Amer. Math. Soc., Providence, RI, 2007.

\bibitem{DiCalg}
J.~Dixmier.
\newblock {\em {$C\sp*$}-algebras}.
\newblock North-Holland Mathematical Library, Vol. 15. North-Holland Publishing Co., Amsterdam-New York-Oxford, 1977.
\newblock Translated from the French by Francis Jellett.

\bibitem{DivNalg}
J.~Dixmier.
\newblock {\em von {N}eumann algebras}, volume~27 of {\em North-Holland Mathematical Library}.
\newblock North-Holland Publishing Co., Amsterdam-New York, french edition, 1981.
\newblock With a preface by E. C. Lance.

\bibitem{VaDo2024}
M.~Donvil and S.~Vaes.
\newblock W*-superrigidity for cocycle twisted group von neumann algebras, 2024.

\bibitem{Enstad22}
U.~Enstad.
\newblock The density theorem for projective representations via twisted group von {N}eumann algebras.
\newblock {\em J. Math. Anal. Appl.}, 511(2):Paper No. 126072, 25, 2022.

\bibitem{Fo2}
G.~B. Folland.
\newblock {\em A course in abstract harmonic analysis}.
\newblock Studies in Advanced Mathematics. CRC Press, Boca Raton, FL, 1995.

\bibitem{GHJ}
F.~M. Goodman, P.~de~la Harpe, and V.~F.~R. Jones.
\newblock {\em Coxeter graphs and towers of algebras}, volume~14 of {\em Mathematical Sciences Research Institute Publications}.
\newblock Springer-Verlag, New York, 1989.

\bibitem{Kiri76}
A.~A. Kirillov.
\newblock {\em Elements of the theory of representations}.
\newblock Grundlehren der Mathematischen Wissenschaften, Band 220. Springer-Verlag, Berlin-New York, 1976.
\newblock Translated from the Russian by Edwin Hewitt.

\bibitem{Kleppner62}
A.~Kleppner.
\newblock The structure of some induced representations.
\newblock {\em Duke Math. J.}, 29:555--572, 1962.

\bibitem{KlepLips1972}
A.~Kleppner and R.~L. Lipsman.
\newblock The {P}lancherel formula for group extensions. {I}, {II}.
\newblock {\em Ann. Sci. \'Ecole Norm. Sup. (4)}, 5:459--516; ibid. (4) 6 (1973), 103--132, 1972.

\bibitem{KlpLips1973}
A.~Kleppner and R.~L. Lipsman.
\newblock Group extensions and the {P}lancherel formula.
\newblock In {\em Harmonic analysis on homogeneous spaces ({P}roc. {S}ympos. {P}ure {M}ath., {V}ol. {XXVI}, {W}illiams {C}oll., {W}illiamstown, {M}ass., 1972)}, volume Vol. XXVI of {\em Proc. Sympos. Pure Math.}, pages 235--238. Amer. Math. Soc., Providence, RI, 1973.

\bibitem{Mackey58}
G.~W. Mackey.
\newblock Unitary representations of group extensions. {I}.
\newblock {\em Acta Math.}, 99:265--311, 1958.

\bibitem{NeuANT}
J.~Neukirch.
\newblock {\em Algebraic number theory}, volume 322 of {\em Grundlehren der mathematischen Wissenschaften [Fundamental Principles of Mathematical Sciences]}.
\newblock Springer-Verlag, Berlin, 1999.
\newblock Translated from the 1992 German original and with a note by Norbert Schappacher, With a foreword by G. Harder.

\bibitem{PetVlt14}
H.~D. Petersen and A.~Valette.
\newblock {$L^2$}-{B}etti numbers and {P}lancherel measure.
\newblock {\em J. Funct. Anal.}, 266(5):3156--3169, 2014.

\bibitem{PlaVla94}
V.~Platonov and A.~Rapinchuk.
\newblock {\em Algebraic groups and number theory}, volume 139 of {\em Pure and Applied Mathematics}.
\newblock Academic Press, Inc., Boston, MA, 1994.
\newblock Translated from the 1991 Russian original by Rachel Rowen.

\bibitem{Rsbg94}
J.~Rosenberg.
\newblock {$C^\ast$}-algebras and {M}ackey's theory of group representations.
\newblock In {\em {$C^\ast$}-algebras: 1943--1993 ({S}an {A}ntonio, {TX}, 1993)}, volume 167 of {\em Contemp. Math.}, pages 150--181. Amer. Math. Soc., Providence, RI, 1994.

\bibitem{Web}
C.~A. Weibel.
\newblock {\em An introduction to homological algebra}, volume~38 of {\em Cambridge Studies in Advanced Mathematics}.
\newblock Cambridge University Press, Cambridge, 1994.

\bibitem{Y22}
J.~Yang.
\newblock Plancherel measures of reductive adelic groups and von neumann dimensions.
\newblock {\em Mathematische Zeitschrift}, 310(20), 2025.

\end{thebibliography}

\textit{E-mail address}: \href{mailto:junyang@fas.harvard.edu}{junyang@fas.harvard.edu}
%\\{Harvard University, Cambridge, MA 02138, USA}

\end{document}